\documentclass[a4paper,10pt]{amsart}
\usepackage{txfonts,amsfonts}
\usepackage[arrow,matrix]{xy}
\usepackage{amsmath,amssymb,amscd,bbm,amsthm,mathrsfs,dsfont,bm}
\usepackage{extarrows}
\usepackage{latexsym}
\usepackage{graphicx}
\usepackage{color}
\usepackage{xy}
\usepackage{leftidx}
\usepackage{chemarrow}
\input xy
\input xypic
\xyoption{all}
\nopagebreak

\allowdisplaybreaks
\newtheorem*{maintheorem}{Theorem}
\newtheorem{theorem}{Theorem}[section]
\newtheorem{lemma}[theorem]{Lemma}

\newtheorem{corollary}[theorem]{Corollary}

\theoremstyle{definition}
\newtheorem{definition}[theorem]{Definition}

\newtheorem{remark}[theorem]{Remark}

\numberwithin{equation}{theorem}

\newtheorem{example}[theorem]{Example}

\DeclareMathOperator{\Hom}{Hom}

\DeclareMathOperator{\uExt}{\underline{Ext}}

\DeclareMathOperator{\uHom}{\underline{Hom}}

\DeclareMathOperator{\id}{id}

\DeclareMathOperator{\Ch}{Ch}
\DeclareMathOperator{\Gr}{Gr}

\DeclareMathOperator{\Cone}{Cone}

\def\bt{\begin{theorem}}
\def\et{\end{theorem}}
\def\bl{\begin{lemma}}
\def\el{\end{lemma}}
\def\br{\begin{remark}}
\def\er{\end{remark}}
\def\bc{\begin{corollary}}
\def\ec{\end{corollary}}

\begin{document}
\title{Ext-algebras of graded skew extensions
}

\author{Y. Shen}

\address{Shen: Department of Mathematics, Zhejiang Sci-Tech University, Hangzhou 310018, China\newline
 \indent \rm{\it E-mail address}:\ yuanshen@zstu.edu.cn}

\author{X. Wang}
\address{Wang: School of Mathematical Sciences, Zhejiang University, Hangzhou 310027, China\newline
 \indent \rm{\it E-mail address}:\ shanshiwangxin@126.com}

\author{G.-S. Zhou}
\address{Zhou: Ningbo Institute of Technology, Zhejiang university, Ningbo 315100, China\newline
 \indent \rm{\it E-mail address}:\ 10906045@zju.edu.cn}

\date{}

\begin{abstract}
In this paper, we study the Ext-algebras of graded skew extensions. For a connected graded algebra $A$ and a graded automorphism $\sigma$, we analyze the Yoneda product of the Ext-algebra of graded skew extension $A[z;\sigma]$, and prove this Ext-algebra is an $R$-smash product of the Ext-algebra of $A$ and the one of polynomial algebra $k[z]$.
\end{abstract}




\maketitle

\section{Introduction}

The Ext-algebra $\uExt_A^*(k,k)$ is a powerful tool to comprehend a connected graded algebra $A$. Equipped with the Yoneda product, Ext-algebras are associative algebras. Such a structure has been discussed, especially for Koszul type algebras (such as \cite{GMMZ,HL,HVZ}). From the view of $A_\infty$-algebras, Ext-algebras carry enough information to recover algebras (\cite{LPWZ3}).


Graded Ore extension is a common constructive method for graded algebras. The usual associative Ext-algebras of Ore extensions have been studied in \cite{Phan} for $\mathcal{K}_2$ algebras. In the study of Artin-Schelter regular algebras and Calabi-Yau algebras, there are many examples coming from graded skew extensions, which are graded Ore extensions with zero derivations. For Koszul algebras, He, Van Oystaeyen and Zhang prove that the Ext-algebras of graded skew extensions are trivial extensions in \cite{HVZ}. The aim of this paper is to give a straightforward description for the Yoneda products of Ext-algebras of graded skew extensions in a general setting.

In fact, graded skew extensions are $R$-smash (twisted tensor) products, namely, tensor vector spaces of algebras with braiding structures. The initial purpose of $R$-smash product is to study factorization and product problems in noncommutative world. In this viewpoint, for a connected graded algebra $A$ and a graded automorphism $\sigma$ of $A$, the graded skew extension $A[z;\sigma]$ can be factorized into $A$ and polynomial algebra $k[z]$. Starting from the free resolutions of trivial modules of $A$ and $k[z]$, we analyze the structure of Ext-algebra of $A[z;\sigma]$ as an associative algebra. The result shows that the Ext-algebra is still an $R$-smash product.

\begin{maintheorem}\label{maintheorem}(Theorem \ref{Ext algebra of Ore extension is an R-smash product})
Let $A$ be a connected graded algebra, and $\sigma$ be an automorphism of $A$. Let $B=A[z;\sigma]$ be a graded skew extension of $A$. Write $E(k[z])$, $E(A)$ and $E(B)$ to be Ext-algebras of $k[z]$, $A$ and $B$, respectively. Then there exists a linear isomorphism $R_E: E(A)\otimes E(k[z])\to E(k[z])\otimes E(A)$ such that as bigraded algebras
$$
E(B)\cong E(k[z])\#_{R_E}E(A)\cong E(A)\#_{R_{E}^{-1}}E(k[z]).
$$
\end{maintheorem}


Here is an outline of the paper. In Section 2, we fix some notations, review definitions of $R$-smash products, and prove a result that Ext-algebras provide a functor from the category of connected graded algebras to the category of bigraded algebras. In Section 3, we focus on the Yoneda product of Ext-algebras of graded skew extensions and give two applications.

Throughout the paper, $k$ is a field and all vector spaces and algebras are over $k$. Unadorned $\otimes$ means $\otimes_k$, and (graded) $k$-dual functor is denoted by $(-)^*$.

\section{Preliminaries}

\subsection{Notations} We give a brief introduction of some notations used in the paper.  We refer \cite{AFH} for more about homological algebras.

Let $A=\bigoplus_{i\in\mathbb{Z}}A_i$ be a graded algebra. We say $A$ is \emph{connected}, if $A_i=0$ for all $i<0$ and $A_0=k$. In this case, we write $\varepsilon_A:A\to k$ for the canonical augmentation of $A$. For a graded automorphism $\sigma$ of $A$, the graded skew extension $A[z;\sigma]$ is obtained by giving the graded polynomial ring $A[z]$ a new multiplication, subject to $za=\sigma(a)z$ for any $a\in A$. We are interested in the case that indeterminant $z$ is of positive degree. This assumption ensures  $A[z;\sigma]$ is connected provided so is $A$.

We denote by $\Gr A$ the category of all graded left $A$-modules with $A$-module homomorphisms of degree zero. For an object $M\in \Gr A$ and an integer $i$, the $i$-th \emph{shift} $M(i)$ is also an object in $\Gr A$  by setting $M(i)_j=M_{i+j}$ for any $j$. Let  $M,N$ be two graded left $A$-modules. For each morphism $f:M\to N$ in $\Gr A$,  there is a morphism from $M(i)$ to $N(i)$ corresponding to $f$, which we still denote by $f$ for simplicity. Write graded vector spaces $\uHom_A(M,N)=\bigoplus_{i\in\mathbb{Z}}\Hom_{\Gr A}(M,N(i))$,  and its derived one $ \uExt^*_A(M,N)$. If $L$ is a right graded $A$-module, tensor product $L\otimes_AM$ is defined to be a graded vector space, where the degree of element $l\otimes m$ is the sum of degrees of homogeneous elements $l$ and $m$.

Let $M$ be a graded left (resp. right) graded $A$-module, and $\nu$ be a graded automorphism of $A$. Then ${^\nu M}$ (resp. ${M^\nu}$) denotes the twisted graded left (resp. right) $A$-module, which equals $M$ as abelian groups with $A$-action $a* m=\nu(a)m$ (resp. $m*a=m\nu(a)$) for all $a\in A$ and $m\in M$.


We denote by $\Ch(\Gr A)$ the category of cochain complexes of graded left $A$-modules with cochain complex morphisms.  For an object $(X^\centerdot,d_X)\in \Ch(\Gr A)$ and  an integer $i$, the \emph{shift} of cochain complex $X[i]^\centerdot$ is a new complex with terms $X[i]^n=X^{n+i}$ and differential $(-1)^id_X$. The endomorphism space $\mathrm{\underline{End}}_A(X^\centerdot)$ is also a cochain complex with $n$-th term $\prod_i\uHom_A(X^i,X^{i+n})$ and differential $d_Xf-(-1)^nfd_X$ for any $f\in \mathrm{\underline{End}}_A(X^\centerdot)^n$. Endowed with the canonical multiplication, $\mathrm{\underline{End}}_A(X^\centerdot)$ is a differential bigraded algebra.  For any graded right $A$-module $L$, we may consider it as a cochain complex concentrated in degree $0$. Then $L\otimes_A X^\centerdot$ is a cochain complex with $L\otimes_A X^n$ as terms and differential $L\otimes_A d_X$.

Let $f:X^\centerdot\to Y^\centerdot$ be a morphism of  cochain complexes in $\Ch(\Gr A)$. The \emph{mapping cone} of $f$ is a cochain complex, denoted by $\Cone(f)$, whose $n$-th term is $X^{n+1}\oplus Y^{n}$  with $n$-th differential
$$
\qquad\qquad\qquad d^n_{\Cone(f)}(x,y)=(-d^{n+1}_X(x),f^n(x)+d^n_Y(y)),\qquad \forall x\in X^{n+1},y\in Y^n.
$$

\subsection{$R$-smash product} $R$-smash product is a kind of solutions for factorization problems, and a more general construction than tensor products. We collect some results of graded version in \cite{CIMZ}.

\begin{definition}
Let $A$ and $B$ be two (bi-)graded algebras, and $R:B\otimes  A\to A\otimes B$ be a (bi-)graded  $k$-linear map. Define $A\#_RB$ to be $A\otimes  B$ as a (bi-)graded vector space with multiplication
$$
m_{A\#_RB}=(m_A\otimes  m_B)(\id_A\otimes R\otimes \id_B),
$$
where $m_A,m_B$ are multiplications of $A,B$ respectively. We say $A\#_RB$ is an \emph{$R$-smash product}, if it is associative and  $1_A\#_R1_B$ is the unit.
\end{definition}

By \cite[Theorem 2.5]{CIMZ}, $R$-smash products always satisfy a \emph{normal} condition, that is, $R(1_B\otimes a)=a\otimes 1_B$ and $R(b\otimes 1_A)=1_A\otimes b$ for any $a\in A,b\in B$. Since the normal condition is a necessary condition for $R$-smash products, we always omit this part for the description of linear maps $R$ in the sequel.

\begin{example}\label{skew extension is smash product}
Let $A$ be a graded algebra, and $\sigma$ be a graded automorphism of $A$. Define a graded linear map
$$
\begin{array}{clll}
R:&k[z]\otimes  A&\to &A\otimes  k[z]\\
&\sum_{i\geq0}c_iz^i\otimes  a&\mapsto&\sum_{i\geq0}\sigma^i(a)\otimes  c_iz^i.
\end{array}
$$
Then $A\#_Rk[z]$ is an $R$-smash product, and the skew extension  $A[z;\sigma]\cong A\#_Rk[z]$.
\end{example}

The following theorem is an important description of $R$-smash products.
\begin{theorem}\label{equivalent conditions of R-smash product}\cite[Theorem 2.10]{CIMZ}
Let $A,B$ and $C$ be (bi-)graded algebras. Then the following two conditions are equivalent
\begin{enumerate}
\item As (bi-)graded algebras, $C\cong A\#_RB$ for some (bi-)graded $k$-linear map $R:B\otimes  A\to A\otimes B$;
\item  there exist (bi-)graded algebra homomorphisms $f_A:A\to C$ and $f_B:B\to C$ such that  $$m_C\circ(f_A\otimes f_B):A\otimes  B\to C$$
 is an isomorphism of (bi-)graded vector spaces, where $m_C$ is the multiplication of $C$.
\end{enumerate}
In this case, $R={\Big (}m_C\circ(f_A\otimes f_B){\Big )}^{-1}\circ m_C\circ(f_B\otimes f_A)$.
\end{theorem}

\subsection{Ext-algebras}

Let $A$ be a connected graded algebra.  There exists a graded free resolution $P^\centerdot$ of ${_Ak}$:
\begin{equation*}
\xymatrix{
P^\centerdot=(\cdots\to A\otimes V_{3}\xlongrightarrow{d^{-3}} A\otimes V_{2}\xlongrightarrow{d^{-2}}A\otimes V_{1}
\xlongrightarrow{d^{-1}}A \to 0 \to\cdots ) \quad \xlongrightarrow{\varepsilon_A} \quad  {_Ak},
}
\end{equation*}
where $V_{n}$ is a graded vector spaces and write $V_0=k$, satisfying $\text{Im}\, d^{-n}\subseteq A_{\geq1}\otimes V_{n-1}$ for all $n=1,2,\cdots$. Such a complex $P^\centerdot$ is called a \emph{minimal free resolution} of ${_Ak}$, which is unique up to isomorphisms in $\Ch(\Gr A)$. It is easy to know that $\uHom_{A}(P^\centerdot,k)$ has zero differential, so the graded vector space $\uExt_A^n(k,k)=\uHom_A(A\otimes V_n, k) \cong V_n^*$ for all $n\geq 0$. Equipped with the Yoneda product, the bigraded vector space
$$
E(A):=\bigoplus_{(n,i)\in\mathbb{Z}^2}\uExt^n_A(k,k)_i
$$
is a bigraded associative algebra, called the \emph{Ext-algebra} of $A$.

In other hand, the cohomology of  differential bigraded algebra $\mathrm{\underline{End}}_A(P^\centerdot)$ is isomorphic to the Ext-algebra $E(A)$. With the help of quasi-isomorphism $\varepsilon_A:P^\centerdot\to {_Ak}$, we have the following quasi-isomorphism
$$
\mathrm{\underline{End}}_A(P^\centerdot)\xlongrightarrow{\uHom_A(P^\centerdot,\varepsilon_A)}   \uHom_A(P^\centerdot,k).
$$

We close this section by a result saying that Ext-algebra arises a contravariant functor from the category of connected graded algebras to the category of bigraded algebras. It maybe a classical result, but we cannot find an appropriate paper to cite. We devote a short space to give a proof.

\begin{theorem}\label{E(-) is a functor}
Taking Ext-algebra $E(-)$ is a contravariant functor from the category of connected graded algebras to the category of bigraded algebras.
\end{theorem}
\begin{proof}
Let $A$ and $B$ be two connected graded algebras, and $f:A\to B$ be a graded algebra homomorphism. The algebra homomorphism $f$ makes any object in $\Gr B$ also an object in $\Gr A$.  Clearly, $\uHom_B(M,N)\subseteq\uHom_A(M,N)$ for all $M,N\in\Gr B$.

Let $P^\centerdot$ and $Q^\centerdot$ be minimal free resolutions of ${_Ak}$ and ${_Bk}$, respectively. And $Q^\centerdot$ is also a cochain complex in $\Ch(\Gr A)$. By Comparison Theorem, there is a morphism of cochain complexes $\widetilde{f}:P^\centerdot\to Q^\centerdot$ in $\Ch(\Gr A)$ lifting $f$ such that $\varepsilon_B\widetilde{f}=\varepsilon_A$. Define a bigraded linear map
$$
E(f):E(B)= \uHom_B(Q^{\centerdot},k)\hookrightarrow \uHom_A(Q^{\centerdot},k)\xlongrightarrow{\uHom_A(\widetilde{f},k)}\uHom_A(P^\centerdot,k)= E(A),
$$
that is, $E(f)(\gamma)=\gamma\circ \widetilde{f}$ for any $\gamma\in E(B)$.

Let $\alpha\in E^i(B)_s$ and $\beta\in E^j(B)_t$ for any $i,j\geq0$ and $s,t\leq0$.  We also consider them as morphisms $\alpha:Q^{\centerdot}\to {_Bk}(s)[i]$ and $\beta:Q^{\centerdot}\to {_Bk}(t)[j]$ be in $\Ch(\Gr B)$. We lift $\alpha$ to a morphism $\widetilde{\alpha}:Q^\centerdot\to Q(s)[i]^\centerdot$ in $\Ch(\Gr B)$ such that $\varepsilon_B[i]\circ \widetilde{\alpha}=\alpha$. Write $\varphi:=E(f)(\alpha):P^\centerdot\to {_Ak}(s)[i]$, and $\widetilde{\varphi}:P^\centerdot\to P(s)[i]^\centerdot$ to be a lift of $\varphi$ satisfying $\varepsilon_A[i]\circ \widetilde{\varphi}=\varphi$ in $\Ch(\Gr A)$. Then we obtain the following commutative diagram in $\Ch(\Gr A)$.
$$
\xymatrix{
Q^\centerdot\ar[rd]^{\alpha}\ar[dd]^{\widetilde{\alpha}}               &             &       P^\centerdot\ar[ll]_{\widetilde{f}}\ar[ld]_{\varphi}\ar[dd]_{\widetilde{\varphi}}\\
                                           &     {_A}k(s)[i]      &          \\
Q(s)[i]^\centerdot\ar[ru]^{\varepsilon_B[i]}            &             &       P(s)[i]^\centerdot\ar[lu]_{\varepsilon_A[i]}\ar[ll]_{\widetilde{f}[i]}
}
$$
 So $
\varepsilon_B[i]\circ(\widetilde{f}[i]\circ\widetilde{\varphi}-\widetilde{\alpha}\circ\widetilde{f})=0.
$
Notice that $\varepsilon_B[i]$ is a quasi-isomorphism and the functor $\uHom_A(P^\centerdot,-)$ preserves quasi-isomorphism, we have that $\widetilde{f}[i]\circ\widetilde{\varphi}-\widetilde{\alpha}\circ\widetilde{f}:P^\centerdot\to Q(s)[i]^\centerdot$ is null homotopic. By the minimality of $P^\centerdot$ and $Q^\centerdot$, we have $\beta[i]\circ\widetilde{f}[i]\circ\widetilde{\varphi}=\beta[i]\circ\widetilde{\alpha}\circ\widetilde{f}$, and
$$
E(f)(\beta)\cdot E(f)(\alpha)=(\beta\circ\widetilde{f})[i]\circ \widetilde{\varphi}=\beta[i]\circ\widetilde{f}[i]\circ \widetilde{\varphi}=\beta[i]\circ\widetilde{\alpha}\circ \widetilde{f}=E(f)(\beta\cdot \alpha).
$$
Hence, $E(f)$ is a bigraded algebra homomorphism.

It remains to show the given rule of associating morphisms for $E(-)$  satisfies the requirement for contravariant functors. It is obvious that $E(\id_A)=\id_{E(A)}$. Let $C$ be another connected graded algebra, and $g:B\to C$ be a graded algebra homomorphism.  Let $F^\centerdot$ be a minimal free resolution of ${_Ck}$. By Comparison Theorem, there is a morphism of cochain complexes $\widetilde{g}:Q^\centerdot\to F^\centerdot$ lifting $g$, satisfying $\varepsilon_C\widetilde{g}=\varepsilon_B$. It is clear $\widetilde{g}\circ\widetilde{f}$ is a morphism of cochain complexes from $P^\centerdot$ to $F^\centerdot$ induced by $g\circ f$, such that  $\varepsilon_C\circ(\widetilde{g}\circ\widetilde{f})=\varepsilon_A$. Therefore, $E(g\circ f)=E(f)\circ E(g)$.
\end{proof}

%

\section{The Yoneda products of Ext-algebras of graded skew extensions}\label{Section Yoneda product of Ext-algebras}
This section focuses on the Yoneda products of Ext-algebras of graded skew extensions. Let $A$ be a connected graded algebra and let $\sigma$ be a fixed graded automorphism of $A$. Let $B=A[z;\sigma]$ be the graded skew extension of $A$, where $\deg z=l\geq1$.

There are two natural graded algebra monomorphisms:
$$
\iota_A:A\to B,\quad\quad \iota_{z}:k[z]\to B,
$$
and two graded algebra epimorphisms:
\begin{equation*}
\begin{array}{cclc}
\pi_A:&B&\to &A\\
~&\sum_{i=0}^ma_iz^i&\mapsto&a_0,
\end{array}\quad
\begin{array}{cclc}
\pi_{z}:&B&\to &k[z]\\
~&\sum_{i=0}^ma_iz^i&\mapsto&\sum_{i=0}^m\varepsilon_A(a_i)z^i,
\end{array}
\end{equation*}
satisfying $\pi_A\circ \iota_A=\id_A, \pi_z\circ \iota_z=\id_{k[z]}$. \begin{corollary}\label{E(A) and E(k[z]) are subalgebras fo E(B)}
The bigraded algebra homomorphisms $E(\pi_A):E(A)\to E(B)$ and $E(\pi_z):E(k[z])\to E(B)$ are both injective.
\end{corollary}
\begin{proof}
By Theorem \ref{E(-) is a functor}, $E(-)$ is a contravariant functor. Then
$$E(\iota_A)\circ E(\pi_A)=\id_{E(A)}\quad\text{ and }\quad E(\iota_B)\circ E(\pi_B)=\id_{E(B)}.$$
Therefore $E(\pi_A)$ and $E(\pi_B)$ are both injective.
\end{proof}

In the viewpoint of $R$-smash products, graded skew extension $B$ can be factored into two subalgebras $A$ and $k[z]$. Following Corollary \ref{E(A) and E(k[z]) are subalgebras fo E(B)}, $E(A)$ and $E(k[z])$ are both bigraded subalgebras of $E(B)$. A natural question is whether $E(B)$ is a smash product of $E(A)$ and $E(k[z])$. We devote the rest of this section to give a positive answer.

There is a classical free resolution of trivial module ${_Bk}$ constructed in \cite{GS,Phan}. We make a little modification to the free resolution for the graded skew extension. Let $P^\centerdot$ be a minimal free resolution of ${_Ak}$:
$$
P^\centerdot=(\cdots\xlongrightarrow{} A\otimes V_m\xlongrightarrow{d^{-m}}A\otimes V_{m-1}\xlongrightarrow{d^{-m+1}}\cdots\xlongrightarrow{}A\otimes V_1
\xlongrightarrow{d^{-1}}A\to 0\to\cdots)\quad \xlongrightarrow{\varepsilon_A}\quad {_Ak},
$$
where each $V_i$ is a graded vector space.

We have a graded $(B,A)$-bimodule homomorphism $\rho_z:B^\sigma(-l)\xlongrightarrow{-\cdot z} B$. Define a cochain complex in $\Ch (\Gr B)$
$$
F^\centerdot:=\Cone\left(\rho_z\otimes_A P^\centerdot:B^\sigma(-l)\otimes_A P^\centerdot\to B\otimes_A P^\centerdot\right).
$$

\begin{lemma}\label{minimal free resolution of Bk}
The morphism of cochain complexes $\varepsilon_B:F^\centerdot\to {_Bk}$ is a quasi-isomorphism. As a consequence, $F^\centerdot$ is a minimal free resolution of $_Bk$.
\end{lemma}
\begin{proof}
One obtains the morphism $\pi_A$ provides a quasi-isomorphism from the mapping cone $\Cone(\rho_z)$ of $B^\sigma(-l)\xlongrightarrow{\rho_z} B$ to $A$ as cochain complexes of graded $(B,A)$-bimodules.  Then $\pi_A\otimes_Ak:\Cone(\rho_z\otimes_Ak)= \Cone(\rho_z)\otimes_Ak\to A\otimes_Ak\cong {_Bk}$ is still a quasi-isomorphism. Now we have the following commutative diagram:
$$
\xymatrix{
B^\sigma(-l)\otimes_A P^\centerdot\ar[r]^(0.55){\rho_z\otimes_A P^\centerdot}\ar[d]^{B^\sigma(-l)\otimes_A\varepsilon_A}_{\simeq}
& B\otimes_A P^\centerdot\ar[r]\ar[d]^{B\otimes_A\varepsilon_A}_{\simeq}
&F^\centerdot\ar@{-->}[d]_{}^{\varphi}\ar[rrd]^{\varepsilon_B}\\
B^\sigma(-l)\otimes_Ak\ar[r]^(0.55){\rho_z\otimes_Ak}  &B\otimes_Ak\ar[r]& \Cone(\rho_z\otimes_Ak)\ar[rr]^{\pi_A\otimes_Ak}&&{_Bk},
}
$$
where $\varphi:=\left(B^\sigma(-l)\otimes_A\varepsilon_A\right)[1]\oplus (B\otimes_A\varepsilon_A)$ is also a quasi-isomorphism. So $\varepsilon_B$ is a quasi-isomorphism. The minimality of $F^\centerdot$ comes from the minimality of $P^\centerdot$ and the action of $\rho_z$.
\end{proof}

In the sequel, we use the minimal free resolution $F^\centerdot$  of $_Bk$ constructed above to bridge the relations among Ext-algebras $E(k[z]), E(A)$ and $E(B)$. Recall that
$$
E(A)\cong \bigoplus_{i\geq0}\uHom_A(P^{-i},k)\cong \bigoplus_{i\geq0} V_i^*,
$$
where  $V_0=k$. By Lemma \ref{minimal free resolution of Bk}, we have $E^0(B)=V_0^*$, and
$$
E^j(B)\cong\uHom_B(F^{-j},k)\cong \uHom_B(B^\sigma(-l)\otimes_AP^{-j+1},k)\oplus\uHom_B(B\otimes_AP^{-j},k)\cong V_{j-1}^*(l)\oplus V_j^*,
$$
for $j\geq1$ as graded vector spaces. We will identify $E^i(A)$ with $V_i^*$ and $E^j(B)$ with $V_{j-1}^*(l)\oplus V_j^*$ for each $i\geq 0,j\geq 1$ in the sequel.

For the polynomial algebra $k[z]$, there is a canonical minimal free resolution of ${_{k[z]}k}$ :
$$
\Cone(\rho'_z):=(0\to k[z](-l)\xlongrightarrow{\rho'_z}k[z]\to 0)\quad \xlongrightarrow{\varepsilon_{k[z]}} \quad {_{k[z]}k},
$$
where $\rho_z'$ is the $k[z]$-module homomorphism of multiplication by $z$. Then
$$E^1(k[z])\cong\uHom_{k[z]}(k[z](-l),k)=\Hom_{\Gr k[z]}(k[z](-l),k(-l))\cong k^*(l)$$
is a one dimensional graded vector space concentrated in degree $-l$, and identified with $k^*(l)$. We fix $\xi$ to be the canonical basis of $E^1(k[z])$ corresponding to the identity map in $k^*(l)$. Note that $V_0=k$.

\begin{remark}\label{a note for E(k[z]), E(A) and E(B)}
In order to make notations uniform, we consider Ext-algebras $E(k[z])$, $E(A)$ and $E(B)$ as graded $k$-dual spaces of some graded vector spaces. However, we need to keep in mind that each element in such identifications for Ext-algebras also has a corresponding representation through minimal free resolutions.

For any homogeneous element $f\in (V^*_i)_{-t}=E^i(A)_{-t}$, there is a corresponding morphism  $\alpha:=P^{\centerdot}\to k(-t)[i]$ in $\Ch(\Gr A)$ induced by $\varepsilon_A\otimes f: A\otimes V_i\to {_Ak}(-t)$.

In other side, for any $g\in (V^*_{i-1}(l))_{-t-l}\subseteq E^i(B)_{-t-l}$, there is a corresponding morphism  $\beta:=F^{\centerdot}\to k(-t-l)[i]$ in $\Ch(\Gr B)$ induced by $B^\sigma(-l)\otimes_A(A\otimes V_{i-1})\xlongrightarrow{\pi_z\otimes_A(A\otimes g)}k[z](-l-t)\xlongrightarrow{\varepsilon_{k[z]}} {_Bk}(-t-l)$. Such representations are useful to prove results below.
\end{remark}



\begin{lemma}\label{The injection E(pi_z)}
$E(\pi_z)(\xi)$ is the identity map of $V_0^*(l)\subseteq E^1(B)$.
\end{lemma}
\begin{proof}
The composition of graded algebra homomorphisms $A\xlongrightarrow{\iota_A} B\xlongrightarrow{\pi_z}k[z]$ makes $k[z]$ a graded $(B,A)$-module, and $k[z]^\sigma\cong k[z]$ as graded $(B,A)$-modules. Let $\mathrm{pr}_0:k[z]\otimes_AP^\centerdot\to k[z]\otimes_A A\cong k[z]$ be the projection to the $0$-th term of the cochain complex $k[z]\otimes_AP^\centerdot$, which is a morphism in $\Ch(\Gr B)$. Then we have a commutative diagram in $\Ch(\Gr B)$ as below
$$
\xymatrix{
B^\sigma(-l)\otimes_AP^\centerdot\ar[r]^(0.55){\rho_z\otimes P^\centerdot}\ar[d]^{\pi_z\otimes_A P^\centerdot}&
B\otimes_A P^\centerdot\ar[d]^{\pi_z\otimes_AP^\centerdot}\ar[r]&
F^\centerdot\ar@{-->}[d]^{\widetilde{\pi_z}}\ar[rdd]^{\varepsilon_B}\\
k[z](-l)\otimes_AP^\centerdot\ar[r]^(0.55){\rho_z\otimes_A P^\centerdot}\ar[d]^{\mathrm{pr}_0}&
k[z]\otimes_AP^\centerdot\ar[r]\ar[d]^{\mathrm{pr}_0}&
\Cone(\rho_z\otimes_AP^\centerdot)\ar@{-->}[d]^{\varphi}\\
k[z](-l)\ar[r]^{\rho'_z}&k[z]\ar[r]&\Cone(\rho'_z)\ar[r]^{\varepsilon_{k[z]}}& {_Bk},
}
$$
where $\widetilde{\pi}_z=(\pi_z\otimes_AP^\centerdot)[1]\oplus (\pi_z\otimes_AP^\centerdot)$, and $\varphi=\mathrm{pr}_0[1]\oplus\mathrm{pr}_0$.

The basis $\xi$ of $E^1(k[z])$ is also related to the morphism $\beta:\Cone(\rho_z')\to k(-l)[1]$ whose $(-1)$-th component is just $\varepsilon_{k[z]}:k[z](-l)\to k(-l)$, by Remark \ref{a note for E(k[z]), E(A) and E(B)}. By Theorem \ref{E(-) is a functor}, one has
$$E(\pi_z)(\beta)=\beta\circ\varphi\circ\widetilde{\pi}_z.$$
However, the nonzero values of $(-1)$-th component of $\varphi$ all takes from $k[z](-l)\otimes_AA$. Hence, the only nonzero part of $E(\pi_z)(\beta)$ equals the composition of following morphisms
$$
\phi:B^{\sigma}(-l)\otimes_AA\xlongrightarrow{\pi_z\otimes_A A}k[z](-l)\otimes_AA\cong k[z](-l)\xlongrightarrow{\varepsilon_{k[z]}} k(-l).
$$
When identifying $E^1(B)$ as $V_0^*(l)\otimes V_1^*$, $\phi$ corresponds to the identity map of $V_0^*(l)$. The result follows.
\end{proof}

\begin{lemma}\label{The injection E(pi_A)}
For any $f\in V_i^*=E^i(A)$, then $E(\pi_A)(f)=f\in V_i^*\subseteq E^i(B)$.
\end{lemma}
\begin{proof}
Consider the following commutative diagram in $\Ch (\Gr B)$:
$$
\xymatrix{
B^\sigma(-l)\otimes_A P^\centerdot\ar[r]^(0.55){\rho_z\otimes_A P^\centerdot}\ar[d]& B\otimes_A P^\centerdot\ar[r]\ar[d]_{\pi_A\otimes_AP^\centerdot}&F^\centerdot\ar@{-->}[d]_{(0,\pi_A\otimes_AP^\centerdot)}\ar[rd]^{\varepsilon_B}\\
0 \ar[r]  &P^\centerdot\ar[r]& P^\centerdot\ar[r]^{\varepsilon_A}&{_Bk}.
}
$$
Similar to the proof of Lemma \ref{The injection E(pi_z)}, the result follows.
\end{proof}


It is the turn to see the Yoneda product of images of $E(A)$ and $E(k[z])$ in $E(B)$.
\begin{lemma}\label{the Yoneda product in E(B)}
Let $f\in (V^*_i)_{-t}=E^i(A)_{-t}$ for $i,t\geq0$.
\begin{enumerate}
\item Then
$$
E(\pi_z)(\xi)\cdot E(\pi_A)(f)=(-1)^if\in \left(V_i^*(l)\right)_{-t-l}\subseteq E^{i+1}(B)_{-t-l}.
$$
\item There exists a bigraded algebra automorphism $\tau$ of $E(A)$ such that
$$
E(\pi_A)(f)\cdot E(\pi_z)(\xi)=\tau(f) \in \left(V_{i}^*(l)\right)_{-t-l}\subseteq E^{i+1}(B)_{-t-l}.
$$
\end{enumerate}
\end{lemma}
\begin{proof}
First of all, we fix corresponding representations along Remark \ref{a note for E(k[z]), E(A) and E(B)}. Write $\alpha:P^\centerdot\to {_Ak}(-t)[i]$ as the corresponding element of $f$ and $\beta:\Cone(\rho_z')\to {_{k[z]}k}(-l)[1]$ as the one of $\xi$ . By Lemma \ref{The injection E(pi_A)}, the image $E(\pi_A)(f)$ in $E^i(B)$ is still an element $f$ in $(V_i^*)_{-t}$. So the morphism $E(\pi_A)(\alpha):F^\centerdot\to {_Bk}(-t)[i]$ in $\Ch(\Gr B)$ induced by
$$
\pi_A \otimes_A(\varepsilon_A\otimes f): B\otimes_A (A\otimes V_i)\to {_Bk}(-t).
$$
By the proof of Lemma \ref{The injection E(pi_z)}, we have the morphism $E(\pi_z)(\beta):F^\centerdot\to {_Bk}(-l)[1]$ in $\Ch(\Gr B)$ is induced by
$$
B^{\sigma}(-l)\otimes_AA\xlongrightarrow{\pi_z\otimes_A A} k[z](-l)\xlongrightarrow{\varepsilon_{k[z]}} {_Bk}(-l).
$$

(a) By Comparison Theorem and the minimality of $P^\centerdot$, there exists a morphism $\widetilde{\alpha}:P^\centerdot\to P^\centerdot(-t)[i]$ such that $\varepsilon_A[i]\circ \widetilde{\alpha}=\alpha$. Note that the $(-i)$-th component of $\widetilde{\alpha}:A\otimes V_i\to A(-t)$ is just $A\otimes f$.

The Yoneda product of $E(\pi_z)(\xi)$ and $E(\pi_A)(f)$ can be read off from the following diagram
$$
\xymatrix{
B^\sigma(-l)\otimes_A P^\centerdot\ar[rr]^(0.55){\rho_z\otimes_A P^\centerdot}\ar[d]^{(-1)^iB^\sigma(-l)\otimes_A \widetilde{\alpha}}
&& B\otimes_A P^\centerdot\ar[r]
\ar[d]^{B\otimes_A \widetilde{\alpha}}
&F^\centerdot\ar@{-->}[d]^{\phi}
\ar[rd]^{E(\pi_A)\left(\alpha\right)}\\
(B^\sigma(-l)\otimes_A P^\centerdot(-t))[i]\ar[rr]^(0.55){\psi}
&&(B\otimes_A P^\centerdot(-t))[i]\ar[r]
&F^\centerdot(-t)[i]\ar[d]^{E(\pi_z)(\beta)[i]}\ar[r]&{_Bk}(-t)[i]
\\
&&&{_Bk}(-t-l)[i+1],&
}
$$
where $\psi:=(-1)^i(\rho_z\otimes_A P^\centerdot(-t))[i]$ and $\phi:=(-1)^i(B^\sigma(-l)\otimes_A \widetilde{\alpha})[1]\oplus (B\otimes_A \widetilde{\alpha})$. Then
$
E(\pi_z)(\beta)\cdot E(\pi_A)(\alpha)=E(\pi_z)(\beta)[i]\circ \phi, $ and the nonezero part equals the following composition
$$
B^\sigma(-l)\otimes_A (A\otimes V_i)\xlongrightarrow{(-1)^iB^{\sigma}(-l)\otimes_A(A\otimes f)}B^\sigma(-l)\otimes_A A(-t)\xlongrightarrow{\pi_z\otimes_A A} k[z](-t-l)\xlongrightarrow{\varepsilon_{k[z]}}{_Bk}(-t-l).$$
Thus, $E(\pi_z)(\xi)\cdot E(\pi_A)\left(f\right)=(-1)^if\in \left(V_i^*(l)\right)_{-t-l}\subset E^{i+1}(B)_{-t-l}$  by Remark \ref{a note for E(k[z]), E(A) and E(B)}.

(b) By Comparison Theorem and the minimality of $P^\centerdot$, there exists a cochain complex isomorphism $
\widetilde{\sigma}:{^{\sigma^{-1}}P^\centerdot}\to P^\centerdot$ in $\Ch (\Gr A)$. Then $\widetilde{\sigma}$ induces a differential bigraded algebra isomorphism $\Psi:\mathrm{\underline{End}}(P^\centerdot)\to \mathrm{\underline{End}}({^{\sigma^{-1}}P^\centerdot})$, and $H\Psi$ is a bigraded algebra automorphism of $E(A)$. Since we identify $E(A)$ with  $\oplus_{i\geq0}V_i^*$, we write $\tau$ to be the automorphism of $\oplus_{i\geq0}V_i^*$.

To be specific, one has the following commutative diagram
$$
\xymatrix{
\mathrm{\underline{End}}_A(P^\centerdot)\ar[rr]^{\uHom_A(P^\centerdot,\varepsilon_A)}\ar[d]^{\Psi} &&\uHom_A(P^\centerdot,k)\ar[d]^{\uHom_A(\widetilde{\sigma},k)}\\
\mathrm{\underline{End}}_A({^{\sigma^{-1}}P^\centerdot})\ar[rr]^(0.47){\uHom_A({^{\sigma^{-1}}P^\centerdot},\varepsilon_A)}&& \uHom_A(^{\sigma^{-1}}P^\centerdot,k).
}
$$
So for any $h\in (V^*_j)_{-s}$, $\tau(h)$ corresponds to a morphism ${^{\sigma^{-1}}P^\centerdot}\to k[j](-s)$ formed by
$$
{^{\sigma^{-1}}(A\otimes V_j)}\xlongrightarrow{\widetilde{\sigma}^{-j}}A\otimes V_j\xlongrightarrow{\varepsilon_A\otimes h}k(-s).
$$

Using different constructions of mapping cones, we have
$$
\xymatrix{
0\ar[r]\ar[d]           &B\otimes_AP^\centerdot\ar[r]\ar@{=}[d]           &B\otimes_A P^\centerdot\ar[d]^{\iota}      \\
B^\sigma(-l)\otimes_AP^\centerdot\ar[r]\ar@{=}[d]&B\otimes_AP^\centerdot\ar[r]\ar[d]&F^\centerdot\ar[d]^{\pi} \\
B^\sigma(-l)\otimes_AP^\centerdot\ar[r]   &0\ar[r]       &(B^\sigma(-l)\otimes_AP^\centerdot)[1].
}
$$

Define a cochain morphism in $\Ch(\Gr B)$ as follows
$$
\theta:F^\centerdot\xlongrightarrow{\pi}(B^\sigma(-l)\otimes_AP^\centerdot)[1]\xlongrightarrow{\cong}(B(-l)\otimes_A{^{\sigma^{-1}} P}^\centerdot)[1]\xlongrightarrow{B(-l)\otimes_A\widetilde{\sigma}}(B(-l)\otimes_AP^\centerdot)[1]\xlongrightarrow{\iota}F^\centerdot(-l)[1],
$$
and it satisfies $E(\pi_z)(\beta)=\varepsilon_{k[z]}[1]\circ\theta$. Then $
E(\pi_A)(\alpha)\cdot E(\pi_z)(\beta)=E(\pi_A)(\alpha)[1]\circ \theta$, and the nonzero part is
$$
B^\sigma(-l)\otimes_A(A\otimes V_i)\cong B(-l)\otimes_A {^{\sigma^{-1}}(A\otimes V_i)}\xlongrightarrow{B(-l)\otimes_A\widetilde{\sigma}^{-i}}B(-l)\otimes_A(A\otimes V_i)\xlongrightarrow{\pi_A\otimes_A(\varepsilon_A\otimes f)} {_Bk}(-t-l).
$$
So  $E(\pi_A)(f)\cdot E(\pi_z)(\xi)=\tau(f) \in \left(V_{i}^*(l)\right)_{-t-l}\subseteq E^{i+1}(B)_{-t-l}$ by Remark \ref{a note for E(k[z]), E(A) and E(B)} and the correspondence of $\tau$.
\end{proof}

Now we can prove the main result.

\begin{theorem}\label{Ext algebra of Ore extension is an R-smash product}
Let $A$ be a connected graded algebra, and $\sigma$ be an automorphism of $A$. Let $B=A[z;\sigma]$ be a graded skew extension of $A$. Write $E(A)$, $E(k[z])$ and $E(B)$ to be Ext-algebras of $A$, $k[z]$ and $B$, respectively. Then as bigraded algebras
$$
E(B)\cong E(k[z])\#_{R_E}E(A)\cong E(A)\#_{R_{E}^{-1}}E(k[z]),
$$
for some $k$-linear normal isomorphism $R_E: E(A)\otimes E(k[z])\to E(k[z])\otimes E(A)$  satisfying 
$$
R_E(f\otimes g)=(-1)^ig\otimes \tau(f).
$$
where  $f\in E^{i}(A)$, $g\in E^1(k[z])$, and $\tau$ is the bigraded automorphism of $E(A)$ in Lemma \ref{the Yoneda product in E(B)}(b).
\end{theorem}
\begin{proof}
By Corollary \ref{E(A) and E(k[z]) are subalgebras fo E(B)}, we have two injections
$$
E(\pi_z):E(k[z])\hookrightarrow E(B),\quad E(\pi_A):E(A)\hookrightarrow E(B).
$$

Write
\begin{eqnarray*}
&&m_1:=m_{E(B)}\left(E(\pi_z)\otimes E(\pi_A)\right):E(k[z])\otimes E(A)\to E(B),\\
&&m_2:=m_{E(B)}\left(E(\pi_A)\otimes E(\pi_z)\right):E(A)\otimes E(k[z])\to E(B),
\end{eqnarray*}
where $m_{E(B)}$ is the Yoneda product of $E(B)$. Clearly, for any $f\in E(A)$, $m_1(1\otimes f)=f$ and $m_2(f\otimes 1)=f$  by Lemma \ref{The injection E(pi_A)}. So $m_1$ and $m_2$ are both isomorphisms of bigraded vector spaces following from Lemma \ref{the Yoneda product in E(B)}.

Using Theorem \ref{equivalent conditions of R-smash product}, there exist $k$-linear isomorphisms of bigraded vector spaces:
$$
\xymatrix{
E(A)\otimes E(k[z])&&  E(k[z])\otimes E(A) \ar@<1ex>[ll]^{(R_E)^{-1}:=m_2^{-1}\circ m_1}\ar@<1ex>[ll];[]^{R_E:=m_1^{-1}\circ m_2}
}
$$
such that $E(B)\cong E(k[z])\#_{R_E} E(A)\cong E(A)\#_{R_E^{-1}} E(k[z])$ as bigraded algebras.

For any $f\in E^i(A)$ and $g\in E^1(k[z])$, by Lemma \ref{the Yoneda product in E(B)}, we have
$$
m_2(f\otimes g)=(-1)^im_1(g\otimes \tau(f)).
$$
Hence, $R_E(f\otimes g)=(-1)^ig\otimes \tau(f)$.
\end{proof}

\begin{example}
We compute the Ext-algebra of quantum plane $k_p[x,y]$ as an easy example, where $p\in k^\times$. Write $A=k[x]$ and $k_p[x,y]\cong A[y;\sigma]$ where $\sigma$ is an automorphism of $A$ by sending $x$ to $px$. The Ext-algebra of $A$ is isomorphic to $k\langle u\rangle/(u^2)$ and $E(k[y])\cong k\langle v\rangle/(v^2)$. Choose $P^\centerdot:0\to A(-1)\xlongrightarrow{-\cdot x}A\to 0$ to be the minimal free resolution of $_Ak$. Then we have a commutative diagram
$$
\xymatrix{
0\ar[r]              &{^{\sigma^{-1}}A}(-1)\ar[r]^{-\cdot x}\ar[d]^{p\cdot\sigma}      &{^{\sigma^{-1}}A}\ar[r]\ar[d]^{\sigma} &0\\
0\ar[r]              &A(-1)\ar[r]^{-\cdot x}          &A\ar[r]     &0.
}
$$
It induces an algebra automorphism of $E(A)$ by sending $u$ to $pu$. By Theorem \ref{Ext algebra of Ore extension is an R-smash product}, the Ext-algebra of quantum plane is
$$
E(k_q[x,y])\cong \left(k\langle u\rangle/(u^2)\right)\,\#_{R_E}\, \left(k\langle v\rangle/(v^2)\right)\cong k\langle u,v \rangle/(u^2,v^2,vu+puv),
$$
where $R_E(v\otimes u)=-pu\otimes v$.
\end{example}

Finally, we give two applications of Theorem \ref{Ext algebra of Ore extension is an R-smash product}.

For a connected graded algebra $A$ generated in degree $1$, it is called a $\mathcal{K}_p$ algebra if $E(A)$ is generated by $\{E^{i}(A)\}_{i=1}^p$ as an associative algebra. In particular, $\mathcal{K}_2$ algebras are the class of algebras defined in \cite{CS}, which is considered as a natural generalization of Koszul algebras. For such a class of algebras, we have an immediate result from Theorem \ref{Ext algebra of Ore extension is an R-smash product}, which contains a special case of \cite[Theorem 1.2]{Phan}.
\begin{corollary}
Let $A$ be a connected graded algebra generated in degree $1$, and $\sigma$ be an automorphism of $A$. Let $B=A[z;\sigma]$ be a graded skew extension of $A$. Then $A$ is a $\mathcal{K}_p$ algebra if and only if $B$ is a $\mathcal{K}_p$ algebra.
\end{corollary}

The other application is about the Frobenius property and Artin-Schelter regular algebras. We say a finite dimensional algebra $E$ is a \emph{Frobenius algebra}, if there exists a  nondegenerate associative bilinear form $\langle-,-\rangle:E\times E\to k$. We refer \cite{Sm} for more details. 


\begin{corollary}\label{Ext preserves Frobenius}
Let $A$ be a connected graded algebra and $\sigma$ be an automorphism of $A$. Let $B=A[z;\sigma]$ be a graded skew extension of $A$. If the Ext-algebra $E(A)$ is a  Frobenius algebra, then Ext-algebra $E(B)$ is also a Frobenius algebra.
\end{corollary}
\begin{proof}
 By Theorem \ref{Ext algebra of Ore extension is an R-smash product}, $E(B)\cong E(k[z])\#_{R_E}E(A)$ for some graded linear map $R_E:E(A)\otimes E(k[z])\to E(k[z])\otimes E(A)$. Let $\tau$ be the  bigraded automorphism of $E(A)$ in Lemma \ref{the Yoneda product in E(B)}.

Clearly $E(k[z])$ is a Frobenius algebra. Write the nondegenerate associative bilinear forms for $E(A)$ and $E(k[z])$ are $\langle-,-\rangle_A$ and $\langle-,-\rangle_z$ respectively. Define a bilinear form $\langle-,-\rangle:E(B)\times E(B)\to k$ satisfying
$$
\langle g_1\#_{R_E} f_1,g_2\#_{R_E} f_2\rangle=
\left\{
\begin{array}{cl}
\langle g_1,g_2\rangle_z\langle f_1,f_2\rangle_A, &\text{if } g_2\in E^0(k[z]),\\
(-1)^i\langle g_1,g_2\rangle_z\langle \tau(f_1),f_2\rangle_A, &\text{if } g_2\in E^1(k[z]),
\end{array}
\right.
$$
where $f_1\in E^i(A),f_2\in E(A)$ and $g_1\in E(k[z])$ for $i>0$.
By straightforward computation, the bilinear form $\langle-,-\rangle$ is  nondegenerate associative. So $E(B)$ is a Frobenius algebra.
\end{proof}

We say a connected graded algebra $A$ is \emph{Artin-Schelter regular} if the global dimension of $A$ equals $d<\infty$, $\uExt_A^i(k,A)=0$ for $i\neq d$ and $\uExt_A^d(k,A)$ is $1$-dimensional. By the approach of $A_\infty$-algebras, we know that $A$ is Artin-Schelter regular if and only if the Ext-algebra $E(A)$ is Frobenius \cite[Corollary D]{LPWZ4}. Now we have a new way to prove graded Ore extension preserving Artin-Schelter regularity.

\begin{corollary}
Let $A$ be an Artin-Schelter regular algebra. Let $\sigma$ be a graded automorphism of $A$ and $\delta$ be a $\sigma$-derivation of $A$. Then graded Ore extension $A[z;\sigma,\delta]$ is Artin-Schelter regular.
\end{corollary}
\begin{proof}
We consider the case of graded skew extension $B=A[z;\sigma]$ firstly. By \cite[Corollary D]{LPWZ4}, one obtains the Ext-algebra $E(A)$ is Frobenius, which implies the Ext-algebra $E(B)$ is also Frobenius by Corollary \ref{Ext preserves Frobenius}. Hence $B$ is a Artin-Scehlter regular algebra.

There is a canonical filtration on $A[z;\sigma,\delta]$ such that the associated graded algebra is the graded skew extension $B$. Then the result follows from \cite[Theorem 3.6]{ZZ1}
\end{proof}

\vskip4.1mm

\noindent {\bf Acknowledgments.} Y. Shen is supported by the NSFC (Grant Nos. 11626215, 11701515) and Science Foundation of Zhejiang Sci-Tech University (Grant No. 16062066-Y); X. Wang is supported by the NSFC (Grant No. 11671351); G.-S. Zhou is supported by the NSFC (Grant No. 11601480).

\end{document}